\documentclass[11pt, parskip=half]{scrartcl}
\usepackage[english]{babel}
\usepackage[utf8]{inputenc}
\usepackage[T1]{fontenc}
\usepackage{lmodern}
\usepackage{mathtools}
\usepackage{amssymb}
\usepackage[thmmarks,amsmath]{ntheorem}
\usepackage[shortlabels]{enumitem}
\usepackage{dsfont}
\usepackage{verbatim}
\usepackage{hyperref}
\usepackage{cleveref}

\theoremseparator{.}
\newtheorem{thm}{Theorem}[section]
\newtheorem{lem}[thm]{Lemma}
\newtheorem{prop}[thm]{Proposition}
\newtheorem{corollary}[thm]{Corollary}

\theorembodyfont{\upshape}
\newtheorem{defn}[thm]{Definition}
\newtheorem{remark}[thm]{Remark}

\newtheorem*{ass}{Assumption}

\theoremstyle{nonumberbreak}
\theoremseparator{:}
\theoremsymbol{\ensuremath{\square}}
\newtheorem{proof}{Proof}

\makeatletter
\newenvironment{cond}[1]{\begin{ass}[#1] \def\@currentlabelname{\upshape\boldmath\textbf{(#1)}}}{\end{ass}}
\makeatother

\newcommand*{\R}{\mathbb{R}}
\newcommand*{\C}{\mathbb{C}}
\newcommand*{\N}{\mathbb{N}}
\newcommand*{\eps}{\varepsilon}
\newcommand*{\cu}{\mathrm{i}}
\newcommand*{\core}{\mathcal{C}}

\newcommand*{\tr}{\operatorname{tr}}
\DeclareMathOperator{\esssup}{ess\,sup}
\DeclareMathOperator{\essinf}{ess\,inf}
\newcommand*{\defeq}{\mathrel{\vcenter{\baselineskip0.5ex \lineskiplimit0pt
			\hbox{\scriptsize.}\hbox{\scriptsize.}}}%
	=}

\numberwithin{equation}{section}

\title{Convergence Rate for Degenerate Partial and Stochastic Differential Equations via weak Poincaré Inequalities}
\author{Alexander Bertram\thanks{Department of Mathematics, TU Kaiserslautern, PO Box 3049, 67653 Kaiserslautern}\ \textsuperscript{,}\thanks{krampe@mathematik.uni-kl.de (corresponding author)} \and Martin Grothaus\footnotemark[1]\ \textsuperscript{,}\thanks{grothaus@mathematik.uni-kl.de}}

\begin{document}
\maketitle

\begin{abstract}
\noindent\small\textsc{Abstract}.
	We employ weak hypocoercivity methods to study the long-term behavior of operator semigroups generated by degenerate Kolmogorov operators	with variable second-order coefficients, which solve the associated abstract Cauchy problem. We prove essential m-dissipativity of the operator, which extends previous results and is key to the rigorous analysis required. We give estimates for the $L^2$-convergence rate by using weak Poincaré inequalities. As an application, we obtain estimates for the (sub-)exponential convergence rate of solutions to the corresponding degenerate Fokker-Planck equations and of weak solutions to the corresponding degenerate stochastic differential equation with multiplicative noise.
\end{abstract}


{\small\textbf{Keywords:}	Degenerate diffusion semigroup, convergence rate, weak hypocoercivity, multiplicative noise, essential m-dissipativity

\section{Introduction and main results}

Hypocoercivity methods go back to Villani (\cite{Villani}), were further developed by Dolbeault, Mouhot and Schmeiser (\cite{DMS09}) as well as Grothaus and Stilgenbauer (\cite{GS14}), and generalized by Grothaus and Wang (\cite{GW19}) to the case with weak Poincaré inequalities. Since then, the latter have been successfully used to derive concrete convergence rates
for solutions to degenerate Fokker-Planck partial differential equations, see for example \cite{GS16,GMS18,BG21}, where last mentioned reference allows for velocity-dependent second-order coefficients. In order to obtain the rate of convergence to their stationary solution and to their equilibrium measure for solutions to degenerate stochastic differential equations, respectively, it is assumed that a Poincaré inequality exists for the measure induced by the potential giving the force. The same is necessary for the complementary result \cite{Wu2001}, which uses a Lyapunov function approach instead. For further complementary results using the Lyapunov function approach see \cite{HM17, BGH21} and references therein. In contrast, the earlier mentioned  framework by Grothaus and Wang is based on weak Poincaré inequalities introduced by Wang and Röckner in \cite{RW01}, which exist under very weak conditions.
This allows for sub-exponential convergence rate estimates, but the applications worked out there still require constant second-order coefficients. We generalize their result to coefficients depending on the second component, while also slightly relaxing the assumptions in the constant case.

We are concerned with the long-term behavior of solutions to the following Itô stochastic differential equation
for $(X_t,Y_t)$ on $E\defeq \R^{d_1}\times\R^{d_2}$, where $d_1,d_2\in\N$ may be different:
\begin{equation}\label{eq:sde}
\begin{aligned}
	\mathrm{d} X_t &= Q\nabla\Psi(Y_t)\,\mathrm{d}t \\
	\mathrm{d} Y_t &= \sqrt{2}\sigma(Y_t)\mathrm{d}B_t - (Q^*\nabla\Phi(X_t)-b(Y_t))\,\mathrm{d}t,
\end{aligned}
\end{equation}
with
\[
	b(y) = (b_i(y))_{1\leq i\leq d_2}
	\quad\text{ and }\quad
	b_i(y)\defeq \sum_{j=1}^{d_2}  \partial_j a_{ij}(y)-a_{ij}(y)\partial_j \Psi(y).
\]
Here $Q$ is a matrix in $\R^{d_1\times d_2}$ with adjoint $Q^*$, $\sigma:\R^{d_2}\to \R^{d_2\times d_2}$ is uniformly non-degenerate with coefficients in $C^1(\R^{d_2})$, and $(a_{ij})_{1\leq i,j\leq d_2}=\Sigma=\sigma\sigma^T$.
Furthermore, the potentials $\Phi \in C^2(\R^{d_1})$, $\Psi\in C^2(\R^{d_2})$ satisfy $Z(\Phi), Z(\Psi) <\infty$,
where $Z(V)\defeq \int_{\R^{d_i}} \mathrm{e}^{-V(x)}\,\mathrm{d}x$.

By applying the Itô transform, we obtain the corresponding Kolmogorov operator $L$, which is given for $f\in C_c^\infty(E)$ by
\begin{equation}
	Lf(x,y) = \tr[\Sigma H_yf] + \langle b(y),\nabla_y f\rangle
				+ \langle Q\nabla\Psi(y),\nabla_x f\rangle - \langle Q^*\nabla\Phi(x),\nabla_y f\rangle,
\end{equation}
where $\langle\cdot,\cdot\rangle$ denotes the Euclidean inner product on $\R^{d_i}$,
$\nabla_x$ and $\nabla_y$ are the gradient operators in the component $x\in\R^{d_1}$ and $y\in\R^{d_2}$ respectively, and $H_yf$ denotes the Hessian matrix of $f$ in the component $y$.
Define the probability measures $\mu_i$ on $(\R^{d_i},\mathcal{B}(\R^{d_i}))$
by $\mu_1\defeq Z(\Phi)^{-1}\mathrm{e}^{-\Phi(x)}\,\mathrm{d}x$ and
$\mu_2\defeq Z(\Psi)^{-1}\mathrm{e}^{-\Psi(y)}\,\mathrm{d}y$ respectively,
then the product measure $\mu\defeq \mu_1\otimes\mu_2$ is a probability measure on the product space $E$.

The aim of this paper is twofold: For one, we show that $(L,\core)$ is essentially m-dissipative on the Hilbert space $X\defeq L^2(E;\mu)$ so that its closure $(L,D(L))$ generates a sub-Markovian $C_0$-semigroup on $X$,
as long as some growth conditions on the potentials and the diffusion coefficients are satisfied. This generalizes previous m-dissipativity results, see \cite{CG10,GN20,BG21}, by allowing for variable second order coefficients and a non-Gaussian measure $\mu_2$ simultaneously. Furthermore it is of central importance to prove our second result, which gives $L^2$-convergence rate estimates of the generated semigroup to a constant under additional conditions.

Besides the application to weak solutions of the SDE \eqref{eq:sde}, we obtain more immediately convergence rate estimates for the solution to the initial value problem
\begin{equation}\label{eq:cauchy-fp}
u(0)=u_0, \qquad
\partial_t u(t) = L^\mathrm{FP}u(t), \quad \text{ for all }t\geq 0,
\end{equation}
to its stationary solution, where $L^\mathrm{FP}$ is given on $C_c^\infty(\R^{d_1})\otimes C_c^\infty(\R^{d_2})$ by
\begin{equation}\label{eq:proto-fokker-planck}
L^\mathrm{FP}f = \sum_{i,j=1}^{d_2} \partial_{y_i}(a_{ij} \partial_{y_j}f + a_{ij}\partial_j\Psi f) -Q\nabla\Psi\cdot\nabla_xf+Q^*\nabla\Phi\nabla_yf,
\end{equation}
see \Cref{sec:fokker-planck} for specifics.

To ensure that the closure of $(L,\core)$ on $X$ generates a $C_0$-semigroup, we need the following conditions:

\begin{cond}{$\Phi$1}\label{ass:x-potential-1}
	The potential $\Phi:\R^{d_1}\to\R$ is assumed to depend only on the first variable $x$, to be bounded from below and locally Lipschitz-continuous, as well as satisfy $Z(\Phi)<\infty$.
\end{cond}
Note that the first part implies $\Phi\in H_\text{loc}^{1,\infty}(\R^{d_1})$. Moreover, $\Phi$ is differentiable $\mathrm{d}x$-a.e. on $\R^{d_1}$,
such that the weak gradient and the derivative of $\Phi$ coincide $\mathrm{d}x$-a.e. on $\R^{d_1}$. In the following, we fix a version
of $\nabla \Phi$.

The potential $\Psi:\R^{d_2}\to\R$ is assumed to depend only on the second variable $y$ and satisfy the following:
\begin{cond}{$\Psi$1}\label{ass:y-potential-1}
	$\Psi$ is measurable, locally bounded, and satisfies $Z(\Psi)<\infty$.
\end{cond}
\begin{cond}{$\Psi$2}\label{ass:y-potential-2}
	It holds that $\Psi\in H_{\mathrm{loc}}^{2,1}(\R^{d_2},\mu_2)$
	as well as $\partial_j\Psi\in L^2_{\mathrm{loc}}(\R^{d_2},\mu_2)$ for $1\leq j\leq d_2$.
\end{cond}
\begin{cond}{$\Psi$3}\label{ass:y-potential-3}
	There are constants $K<\infty$ and $\alpha\in[1,2)$ such that
	\[
		|\nabla^2\Psi|\leq K(1+|\nabla\Psi|^\alpha),
	\]
	where $\nabla^2\Psi$ denotes the Hessian matrix of $\Psi$.
\end{cond}

The second-order coefficient matrix $\Sigma$ satisfies:
\begin{cond}{$\Sigma$1}\label{ass:ellipticity}
	$\Sigma$ is symmetric and uniformly strictly elliptic, i.e.~there is some $0<c_\Sigma<\infty$ such that
		\[
			\langle v,\Sigma(y) v\rangle \geq c_\Sigma^{-1}\cdot |v|^2\quad\text{ for all }v\in\R^{d_2} \text{ and $\mu_2$-almost all }y\in\R^{d_2}.
		\]
\end{cond}
\begin{cond}{$\Sigma$2}\label{ass:coeff-derivatives}
	For each $1\leq i,j\leq d_2$, $a_{ij}$ is bounded and locally Lipschitz-continuous, which implies $a_{ij}\in H_\text{loc}^{1,\infty}(\R^{d_2})$.
\end{cond}
\begin{cond}{$\Sigma$3}\label{ass:coeff-growth}
	There are constants $0\leq M<\infty$, $0\leq \beta <1$ such that for all $1\leq i,j,k\leq d_2$
	\[
		|\partial_k a_{ij}(y)|\leq M(\mathds{1}_{B_1(0)}(y)+|y|^\beta)\quad\text{ for $\mu_2$-almost all }y\in\R^{d_2}.
	\]
\end{cond}

Based on the last assumption, we introduce one more restriction on the growth behavior of $\Phi$:
\begin{cond}{$\Phi$2}\label{ass:x-potential-2}
	Assume that $|\nabla\Phi|\in L^2(\R^{d_1},\mu_1)$ and that there is some $N<\infty$ such that
	\[
		|\nabla\Phi(x)|\leq N(1+|x|^\gamma),\quad\text{ where }\gamma<\frac{1}{\beta},
	\]
	in case that $\beta$ being provided by \nameref{ass:coeff-growth} is strictly positive.
\end{cond}

\begin{thm}\label{thm:ess-m-diss}
	Let \nameref{ass:ellipticity}-\nameref{ass:coeff-growth}, \nameref{ass:y-potential-1}-\nameref{ass:y-potential-3} and \nameref{ass:x-potential-1}-\nameref{ass:x-potential-2} be fulfilled.
	Then the linear operator $(L,C_c^\infty(\R^{d_1})\otimes C_c^\infty(\R^{d_2}))$ is essentially m-dissipative
	and hence its closure $(L,D(L))$ generates a strongly continuous contraction semigroup $(T_t)_{t\geq 0}$ on $X$.
	Moreover, $(T_t)_{t\geq 0}$ is sub-Markovian, conservative, and possesses $\mu$ as an invariant measure,
	i.e.~$\int_E T_tf\,\mathrm{d}\mu=\int_E f\,\mathrm{d}\mu$ for all $f\in X$.
\end{thm}

We now present our final result on the convergence rate of the generated semigroup.
Note that the assumptions in the following theorem are stronger than those made above,
so that \Cref{thm:ess-m-diss} applies.

\begin{thm}\label{thm:main-result}
	Let $QQ^*$ be invertible and let $\Sigma$ be uniformly strictly elliptic with each $a_{ij}$ being bounded and locally Lipschitz.
	Let further $M<\infty$, $0\leq\beta<1$ and $1<p_\Sigma\leq\infty$ such that for all $1\leq i,j,k\leq d_2$, the following holds:
	\[
	|\partial_k a_{ij}(y)|\leq M(\mathds{1}_{B_1(0)}(y)+|y|^\beta)
	\quad\text{ and }\quad
	|\partial_k a_{ij}(y)|\in L^{2p_\Sigma}(\mu_2).
	\]
	Let $\Phi\in C^2(\R^{d_1})$ bounded from below with $Z(\Phi)<\infty$,
	and let there be a constant $C<\infty$ such that
	\[
		|\nabla^2\Phi|\leq C(1+|\nabla\Phi|).
	\]
	If $\beta>0$, let further $N<\infty$, $0\leq\gamma<\frac1\beta$ such that
	$|\nabla\Phi(x)|\leq N(1+|x|^\gamma)$.
	
	Let $\psi\in C^3([0,\infty))$, $K<\infty$, $1\leq\alpha<2$, $\Lambda\in\R^{d_2\times d_2}$ invertible, $a\in\R^{d_2}$ such that 
	\[
		\Psi(y)=\psi(|\Lambda y-a|^2),
		\quad Z(\Psi)<\infty,
		\quad\text{ and }\quad
		|\nabla^2\Psi|\leq K(1+|\nabla\Psi|^\alpha).
	\]
	Assume further that $\partial_i\partial_j\partial_k\Psi\in L^2(\mu_2)$ for all $1\leq i,j,k\leq d_2$.
	
	Then there exist decreasing functions $\alpha_\Phi,\alpha_\Psi:(0,\infty)\to [1,\infty)$ and constants $0<c_1,c_2<\infty$ such that
	\begin{align*}
		\mu_1(f^2)-\mu_1(f)^2 &\leq \alpha_\Phi(r)\mu_1(|\nabla_x f|^2)+r\|f\|_\mathrm{osc}^2,
			\quad r>0,f\in C_b^1(\R^{d_1}), \\
		\mu_2(f^2)-\mu_2(f)^2 &\leq \alpha_\Psi(r)\mu_2(|\nabla_y f|^2)+r\|f\|_\mathrm{osc}^2,
			\quad r>0,f\in C_b^1(\R^{d_2}),
	\end{align*}
	and
	\begin{equation}\label{eq:conv-estimate}
		\mu((T_tf)^2)-\mu(T_tf)^2
		\leq c_1\xi(t)\|f\|_{\mathrm{osc}}^2,
		\qquad \text{ for all } t\geq 0,\ f\in L^\infty(\mu),
	\end{equation}
	where
	\begin{equation}\label{eq:conv-rate}
		\xi(t)\defeq \inf\left\{ r>0: c_2t\geq \alpha_1(r)^2\alpha_2\left(\frac{r}{\alpha_1(r)^2}\right)
		\log\left( \frac1{r}\right)  \right\}.
	\end{equation}
\end{thm}
\begin{remark}
	\begin{enumerate}
		\item If $d_1=d_2$ and $\Psi(y)=\frac12 |y|^2$, i.e. the measure $\mu_2$ is the standard Gaussian measure on $\R^{d_2}$, then \Cref{thm:ess-m-diss} is a slight improvement of \cite[Theorem~4.5]{BG21} due to a better trade-off between $\Sigma$ and $\Phi$. In that case, we get $\alpha_2\equiv c_\Sigma$ in \eqref{eq:conv-rate}. In particular, the convergence is exponential if there is a (strong) Poincaré inequality for $\Phi$.
		
		\item If $\Sigma=I$, then \Cref{thm:main-result} yields the same convergence result as \cite{GW19}.
			However, we show that $\mu_2(|\nabla_y\Psi|^4)<\infty$ need not be assumed and give an alternative condition for $\Psi$ by replacing
			\[
				\sup_{r\geq 0} \left| \psi'(r)+2r\psi''(r)-\frac{2r\psi'''(r)+(d_2+2)\psi''(r)}{\psi'(r)} \right| < \infty
			\]
			with $L^2(\mu_2)$-integrability of the third-order derivatives of $\Psi$.
			
		\item Since the convergence rate $\xi$ depends from the choice of $\Sigma$ only via the ellipticity constant $c_\Sigma$ as a factor in front of $\alpha_2$, we refer to \cite{GW19} for a detailed list of explicit rates depending on the chosen potentials. We list two examples for illustration purposes:
		\begin{enumerate}[(i)]
			\item Let $\Phi(x)=k(1+|x|^2)^{\delta/2}+g_1(x)$ and $\Psi(y)=\kappa(1+|y|^2)^{\eps/2}+g_2(y)$,
				where $k,\kappa,\delta,\eps>0$ are constant and $g_1\in C_b^2(\R^{d_1})$, $g_2 \in C_b^3(\R^{d_2})$ with $g_2(y)=\eta(|y|^2)$ for some $\eta\in C^3([0,\infty))$.
				Then, \eqref{eq:conv-estimate} holds with
				\[
					\xi(t)=\exp(-c_2 t^{\omega(\delta,\eps)}),
					\text{ where }
					\omega(\delta,\eps)=\frac{\delta\eps}{\delta\eps+8\eps(1-\delta)^{+}+4\delta(1-\eps)^{+}},
				\]
				for some constants $0<c_1,c_2<\infty$. In particular, if $\delta,\eps\geq 1$, then the decay rate is exponential.
			\item Let $d\defeq d_1=d_2$, $\Phi(x)=\frac{q+d}2 \log(1+|x|^2)+g_1(x)$ for some $q>0$ and
				$\Psi(y)=\frac{p+d}{2}\log(1+|y|^2)+g_2(y)$ for some $p>0$,
				where $g_1\in C_b^2(\R^{d})$ and $g_2 \in C_b^3(\R^{d})$
				with $g_2(y)=\eta(|y|^2)$ for some $\eta\in C^3([0,\infty))$.
				Then, \eqref{eq:conv-estimate} holds with
				\[
				\begin{aligned}
					\xi(t) &= c_2(1+t)^{-\omega(p,q)}(\log(\mathrm{e}+t))^{\omega(p,q)},
					&& \text{ where }\\
					\omega(p,q) &= \frac1{2\theta(q)+\theta(p)+2\theta(q)\theta(p)}
					&& \text{ and }\\
					\theta(r) &= \frac{d+r+2}{r} \wedge \frac{4r+4+2d}{(r^2-4-2d-2r)^{+}},
				\end{aligned}
				\]
				for some constants $0<c_1,c_2<\infty$.
		\end{enumerate}
	\end{enumerate}
\end{remark}

\section{Essential m-dissipativity of the operator \texorpdfstring{$L$}{L}}

\subsection{Preliminaries}

Throughout this section, we assume \nameref{ass:ellipticity}-\nameref{ass:coeff-growth}, \nameref{ass:y-potential-1}-\nameref{ass:y-potential-3} and \nameref{ass:x-potential-1}-\nameref{ass:x-potential-2} to be satisfied.

\begin{defn}\label{def:coeff-const}
	Depending on the choice of $\Sigma$, we introduce the following constants:
	\begin{align*}
		M_\Sigma &\defeq \max\{ \|a_{ij}\|_\infty: 1\leq i,j \leq d_2 \},\\
		B_\Sigma &\defeq \max\{ |\partial_j a_{ij}(y)|: y\in \overline{B_1(0)},\ 1\leq i,j \leq d_2 \} \qquad\text{ and } \\
		N_\Sigma &\defeq \sqrt{M_\Sigma^2+B_\Sigma^2+d_2M^2}.
	\end{align*}
\end{defn}

We introduce the following useful decomposition of $L$, which we require later to apply the weak hypocoercivity method:
\begin{defn}\label{def:operators}
	Let $\mathcal{C}\defeq C_c^\infty(\R^{d_1})\otimes C_c^\infty(\R^{d_2})$ be the tensor product space of smooth compactly supported functions on $\R^{d_1}$ and $\R^{d_2}$.
	We define the linear operators $S$ and $A$ for $f\in\core$ by
	\begin{align*}
		Sf 	&\defeq \sum_{i,j=1}^{d_2} a_{ij}\partial_{y_j}\partial_{y_i}f + \sum_{i=1}^{d_2} b_i\partial_{y_i}f, \\
			&\quad\text{ where } b_i(y)= \sum_{j=1}^{d_2}(\partial_j a_{ij}(y)-a_{ij}(y)\partial_j\Psi(y)),\\
		Af &\defeq  Q^*\nabla\Phi\cdot\nabla_y f - Q\nabla\Psi\cdot\nabla_x f.
	\end{align*}
	
	Integration by parts shows that $(S,\mathcal{C})$ is symmetric and non-positive definite on $X$, and $(A,\mathcal{C})$ is antisymmetric on $X$, which implies both are dissipative and therefore closable.
	We denote their closure respectively by $(S,D(S))$ and $(A,D(A))$
	Since $(L,\core)=(S-A,\core)$, it is also dissipative and closable, and we denote its closure by $(L,D(L))$.
\end{defn}

We observe that for $f\in \mathcal{C}$ and $g\in H^{1,2}(E,\mu)$, integration by parts yields
\begin{equation}\label{eq:gradient-form}
(Lf,g)_X = -\int_E \left\langle \nabla f, \begin{pmatrix}0 & -Q\\Q^* & \Sigma\end{pmatrix}\nabla g \right\rangle\, \mathrm{d}\mu.
\end{equation}

We also provide the following $L^p$-integrability result, which is based on and vastly generalizes a Lemma from Villani, see \cite[Lemma A.24]{Villani}
\begin{thm}\label{thm:arbitrary-integrability}
	Let $V\in C^1(\R^d)\cap H_{\mathrm{loc}}^{2,\infty}(\R^d)$ satisfy \nameref{ass:y-potential-3} with $\Psi=V$ and let $k\in\N$.
	Then there is a constant $C_k<\infty$ such that for all $g\in H^{1,2k}(\mathrm{e}^{-V}\mathrm{d}x)$,
	the following holds:
	\[
		\int_{\R^d}|\nabla V|^{2k}g^{2k}\mathrm{e}^{-V}\,\mathrm{d}x
		\leq C_k\left( \int_{\R^d}g^{2k}\mathrm{e}^{-V}\,\mathrm{d}x
		+ \int_{\R^d}|\nabla g|^{2k}\mathrm{e}^{-V}\,\mathrm{d}x \right).
	\]
\end{thm}
\begin{proof}
	The proof is derived from the proof of \cite[Lemma A.24]{Villani}, which shows the statement for $\alpha=1$ and $k=1$. We therefore assume $\alpha>1$, but also only show the proof for $k=1$, as the other cases follow analogously after adjusting the constant.
	Without loss of generality, we assume that $g\in C_c^\infty(\R^{d})$, since that space is dense in $H^{1,2k}(\mathrm{e}^{-V(x)}\,\mathrm{d}x)$ (which can be proven analogously to the well-known denseness result in $H^{m,p}(\Omega)$ for $\Omega$ satisfying the segment property, since $\mathrm{e}^{-V(x)}\,\mathrm{d}x$ is locally equivalent to the Lebesgue measure).
	
	Let the measure $\mathrm{e}^{-V}\mathrm{d}x$ be denoted by $\mu_V$. Then, integration by parts yields
	\[
	\int_{\R^d}|\nabla V|^2g^2\,\mathrm{d}\mu_V
	= -\int_{\R^d}g^2\langle\nabla V, \nabla(\mathrm{e}^{-V})\rangle\,\mathrm{d}x
	= \int_{\R^d}g^2\Delta V \,\mathrm{d}\mu_V
	+ \int_{\R^d}2g\langle\nabla V, \nabla g\rangle \,\mathrm{d}\mu_V.
	\]
	Due to \nameref{ass:y-potential-3}, it holds that
	\[
	|\Delta V|=\sqrt{(\Delta V)^2}\leq \sqrt{d}\,|\nabla^2 V|\leq \sqrt{d}K(1+|\nabla V|^\alpha).
	\]
	This with Cauchy–Bunyakovsky–Schwarz implies
	\[
	\int_{\R^d}g^2\Delta V \,\mathrm{d}\mu_V
	\leq \sqrt{d}K\left(\int_{\R^d}g^2\,\mathrm{d}\mu_V
	+\sqrt{\int_{\R^d}g^2|\nabla V|^2 \,\mathrm{d}\mu_V}
	\sqrt{\int_{\R^d}g^2|\nabla V|^{2(\alpha-1)} \,\mathrm{d}\mu_V} \right),
	\]
	and Young's inequality further gives
	\[
	\int_{\R^d}g^2\Delta V \,\mathrm{d}\mu_V
	\leq \sqrt{d}K\int_{\R^d}g^2\,\mathrm{d}\mu_V
	+\frac14 \int_{\R^d}|\nabla V|^2g^2\,\mathrm{d}\mu_V
	+dK^2 \int_{\R^d}|\nabla V|^{2(\alpha-1)}g^2\,\mathrm{d}\mu_V.
	\]
	Similarly, it holds that
	\[
	\begin{aligned}
	2\int_{\R^d}g\langle\nabla V, \nabla g\rangle \,\mathrm{d}\mu_V
	&\leq 2\sqrt{\int_{\R^d}g^2|\nabla V|^2 \,\mathrm{d}\mu_V}
	\sqrt{\int_{\R^d}|\nabla g|^2 \,\mathrm{d}\mu_V}\\
	&\leq \frac14 \int_{\R^d}|\nabla V|^2g^2\,\mathrm{d}\mu_V
	+4\int_{\R^d}|\nabla g|^2 \,\mathrm{d}\mu_V.
	\end{aligned}
	\]
	Adding these estimates together and subtracting $\frac12\int_{\R^{d}} |\nabla V|^2g^2\,\mathrm{d}\mu_V$ from both sides, we obtain
	\[
	\int_{\R^d}|\nabla V|^2g^2\,\mathrm{d}\mu_V
	\leq 2\sqrt{d}K\int_{\R^d}g^2\,\mathrm{d}\mu_V
	+8\int_{\R^d}|\nabla g|^2 \,\mathrm{d}\mu_V
	+2dK^2\int_{\R^d}|\nabla V|^{2(\alpha-1)}g^2\,\mathrm{d}\mu_V.
	\]
	If $2(\alpha-1)\leq 1$, then $|\nabla V|^{2(\alpha-1)}\leq1+|\nabla V|$, and therefore
	\[
	2dK^2\int_{\R^d}|\nabla V|^{2(\alpha-1)}g^2\,\mathrm{d}\mu_V
	\leq 2dK^2\int_{\R^d}g^2\,\mathrm{d}\mu_V
	+2d^2K^4 \int_{\R^d}g^2\,\mathrm{d}\mu_V
	+\frac12 \int_{\R^d}|\nabla V|^2g^2\,\mathrm{d}\mu_V,
	\]
	hence
	\[
	\int_{\R^d}|\nabla V|^2g^2\,\mathrm{d}\mu_V
	\leq 4(\sqrt{d}K+dK^2+d^2K^4)\int_{\R^d}g^2\,\mathrm{d}\mu_V
	+16\int_{\R^d}|\nabla g|^2 \,\mathrm{d}\mu_V,
	\]
	which satisfies the claimed inequality.
	In the remaining case that $2(\alpha-1)\in(1,2)$, we obtain instead
	\[
	2dK^2\int_{\R^d}|\nabla V|^{2(\alpha-1)}g^2\,\mathrm{d}\mu_V
	\leq 2d^2K^4\int_{\R^d}|\nabla V|^{2(2(\alpha-1)-1)}g^2\,\mathrm{d}\mu_V
	+\frac12 \int_{\R^d}|\nabla V|^2g^2\,\mathrm{d}\mu_V.
	\]
	Repeating this procedure iteratively provides us with estimates of the form
	\[
	\begin{aligned}
	\int_{\R^d}|\nabla V|^2g^2\,\mathrm{d}\mu_V
	&\leq 2^{m+2}\int_{\R^d}|\nabla g|^2 \,\mathrm{d}\mu_V
	+2^{m}\sqrt{d}K\int_{\R^d}g^2\,\mathrm{d}\mu_V\\
	&\qquad	+(2dK^2)^{2^{(m-1)}}\int_{\R^d}|\nabla V|^{\gamma_m(\alpha)}g^2\,\mathrm{d}\mu_V,
	\end{aligned}
	\]
	where $\gamma_m(\alpha)=2^m(\alpha-2)+2$.
	Since there is some $\eps>0$ such that $\alpha=2-\eps$,
	we can write $\gamma_m(\alpha)=2-2^m\eps$, so there is some $m_0\in\N$ such that $0<\gamma_{m_0}(\alpha)\leq 1$. Then the claim follows as above for the case $2(\alpha-1)\leq 1$.
\end{proof}

\begin{remark}\label{rem:psi-lp}
	The conditions \nameref{ass:y-potential-1}--\nameref{ass:y-potential-3} together imply that $\Psi\in C^1(\R^{d_2})$ and that $\nabla\Psi$ is locally Lipschitz-continuous. This can be proven using Sobolev embedding and interpolation theory, see \cite[Remark 2.4(ii)]{GN20} or \cite[Theorem A6.2]{Conrad2011}. Therefore, we can apply \Cref{thm:arbitrary-integrability} with $g\equiv 1$ to obtain $|\nabla\Psi|\in L^p(\mu_2)$ for all $1\leq p<\infty$. Due to \nameref{ass:y-potential-3}, this also implies $|\nabla^2\Psi|\in L^p(\mu_2)$ for all $1\leq p<\infty$, where $\nabla^2\Psi$ denotes the Hessian matrix.
\end{remark}

\subsection{Essential m-dissipativity for Lipschitz coefficients and potential}

We begin with an essential self-adjointness result for the symmetric part $(S,\core)$.
Since none of the coefficients depend on the coordinate $x$, we first reduce it to an operator on $L^2(\R^{d_2};\mu_2)$.
\begin{lem}\label{lem:ess-self-adjoint-y}
	The differential operator $(\tilde{S},C_c^\infty(\R^{d_2}))$, where
	\[
	\tilde{S}f=\sum_{i,j=1}^d a_{ij}\partial_j\partial_if + \sum_{i=1}^d \sum_{j=1}^d(\partial_j a_{ij}-a_{ij}\partial_j\Psi)\partial_if   \quad\text{ for }f\in C_c^\infty(\R^{d_2}),
	\]
	is essentially self-adjoint on $L^2(\R^{d_2},\mu_2)$.
\end{lem}
\begin{proof}
	We apply \cite[Theorem 4.5]{BG21} with $A=\Sigma$ and $\rho=\mathrm{e}^{-\Psi}$ to show essential self-adjointness. For this we need $\mathrm{e}^{\Psi}$ to be locally bounded, which is satisfied due to \nameref{ass:y-potential-1}. The conditions on $A$ are satisfied due to \nameref{ass:ellipticity}--\nameref{ass:coeff-derivatives}, and sufficient regularity of $\rho$ follows from \Cref{rem:psi-lp}.
\end{proof}

\begin{corollary}
	The symmetric operator $(S,\core)$ as defined in \Cref{def:operators} is essentially self-adjoint on $X=L^2(E;\mu)$.
\end{corollary}
\begin{proof}
	let $g=g_1\otimes g_2\in \core$ be a pure tensor.
	Then, due to \Cref{lem:ess-self-adjoint-y}, there is a sequence $(\tilde{f}_n)_{n\in\N}$ in $C_c^\infty(\R^{d_2})$ such that $(I-\tilde{S})\tilde{f}_n\to g_2$ in $L^2(\R^{d_2},\mu_2)$ as $n\to\infty$. Define $f_n \in \core$ for each $n\in\N$ as $f_n\defeq g_1\otimes \tilde{f}_n$.
	Then
	\[
		\|(I-S)f_n-g\|_X
		= \| g_1\otimes ((I-\tilde{S})\tilde{f}_n - g_2)\|_X
		= \|g_1\|_{L^2(\mu_1)}\cdot \|(I-\tilde{S})\tilde{f}_n-g_2 \|_{L^2(\mu_2)},
	\]
	which converges to zero as $n\to\infty$.
	By taking linear combinations, this shows that $(I-S)(\core)$ is dense in $X$, since $\core$ is dense in $X$. Therefore, the dissipative operator $(S,\core)$ is essentially self-adjoint.
\end{proof}

Since $S$ is dissipative on $D_0\defeq L_c^2(\mu_1)\otimes C_c^\infty(\R^{d_2})\supset \mathcal{C}$,
the operator $(S,D_0)$ is essentially m-dissipative as well. We introduce the unitary transformations
\begin{align}\label{eq:unitary-transforms}
	U:H &\to L^2(E,\mathrm{d}(x,y)), && f\mapsto \sqrt{Z(\Psi)^{-1}}\,\mathrm{e}^{-\frac12(\Phi+\Psi)} \qquad\text{ and }\\
	U_\Psi: L^2(\mu_2) &\to L^2(\R^{d_2},\mathrm{d}y), && f\mapsto \sqrt{Z(\Psi)^{-1}}\,\mathrm{e}^{-\frac12\Psi}
\end{align}
as well as the subspace $D_1\defeq L_c^2(\R^{d_1},\mathrm{d}x)\otimes U_\Psi C_c^\infty(\R^{d_2})$ of $L^2(E,\mathrm{d}(x,y))$.
Note that due to \nameref{ass:x-potential-1}, $\mathrm{e}^{-\Phi}$ is strictly positive and locally bounded, which implies that $L_c^2(\R^{d_1},\mathrm{d}x)$ and $L_c^2(\mu_1)$ coincide. Hence we obtain $D_1=UD_0$ and essential m-dissipativity of $(L_0,D_1)$, where $L_0=USU^{-1}$. For $f\in D_1$, we obtain the representation
\[
	L_0f	= \sum_{i,j=1}^d a_{ij}\partial_{y_j}\partial_{y_i}f
			- \frac14(\nabla\Psi,\Sigma \nabla\Psi)f
			+ \frac12 \sum_{i,j=1}^{d_2} a_{ij}\partial_j\partial_i\Psi f
			+ \sum_{i,j=1}^d \partial_j a_{ij} (\tfrac12\partial_i\Psi f+\partial_{y_i}f), 
\]
where the differential operators $\nabla$ and $\partial_i$ are understood in the distributional sense.

Next we perturb $(L_0,D_1)$ by the multiplication operator $(A_1,D_1)$ given by the measurable function $\cu Q\nabla\Psi\cdot x:E\to\C$.
Clearly $(A_1,D_1)$ is well-defined in $L^2(E,\mathrm{d}(x,y))$ and antisymmetric, hence dissipative.
\begin{prop}\label{prop:no-potential}
	The operator $(L_1,D_1)$ defined by $L_1\defeq L_0+A_1$ is essentially m-dissipative on $ L^2(E,\mathrm{d}(x,y))$.
\end{prop}
\begin{proof}
We introduce the complete orthogonal family of projections $P_n$ defined via $P_nf\defeq \xi_nf$,
where $\xi_n$ is given by $\xi_n(x,y) = \mathds{1}_{[n-1,n)}(|x|)$.
Each $P_n$ leaves $D_1$ invariant and commutes with both $L_0$ and $A_1$.
We have to show that each $P_nA_1$ is $P_nL_0$-bounded with Kato bound zero. Let $f\in P_nD_1$. Then it holds that
\[
	\|\cu Q\nabla\Psi\cdot xf\|_{L^2}^2\leq n^2|Q|_2^2\int_E |\nabla\Psi|^2f^2\,\mathrm{d}(x,y).
\]
Hence, it is enough to show that there are finite constants $a,b$ such that
\begin{equation}\label{eq:Kato-bound}
	\int_E |\nabla\Psi|^2f^2\,\mathrm{d}(x,y)\leq a(L_0f,f)+b\|f\|_{L^2}\quad\text{ for all }f\in D_1.
\end{equation}
We get
\[
\begin{aligned}
	\int_E |\nabla\Psi|^2f^2\,\mathrm{d}(x,y) &\leq 4c_\Sigma \int_E \frac14\langle \nabla\Psi,\Sigma \nabla\Psi\rangle f^2\,\mathrm{d}(x,y) \\
	&\leq 4c_\Sigma \left( \int_E \frac14\langle \nabla\Psi,\Sigma \nabla\Psi\rangle f^2\,\mathrm{d}(x,y)
			+ \int_E \langle \nabla_y f,\Sigma \nabla_yf\rangle\,\mathrm{d}(x,y)  \right) \\
	&= 4c_\Sigma \left( (-L_0f,f)_{L^2}+\int_E\frac12 \sum_{i,j=1}^{d_2} (a_{ij}\partial_j\partial_i\Psi + \partial_j a_{ij} \partial_i\Psi) f^2\,\mathrm{d}(x,y) \right)
\end{aligned}
\]
Let $R_1\defeq 4c_\Sigma$ and recall that due to \nameref{ass:coeff-growth} with $\beta=0$, it holds that $|\partial_j a_{ij}|\leq 2M$. Using the Hölder and Young inequalities for $p=q=2$, it follows that
\[
\begin{aligned}
	\frac{R_1}2\left|\int_E \sum_{i,j=1}^{d_2} \partial_j a_{ij}\partial_i\Psi f^2\,\mathrm{d}(x,y)\right|
	&\leq \frac{R_1}2\sum_{i=1}^{d_2}\left\| \left(\sum_{j=1}^{d_2}\partial_j a_{ij} \right)fR_1^{1/2} \right\|_{L^2} \cdot \|R_1^{-1/2}\partial_i\Psi f\|_{L^2} \\
	&\leq \frac14\sum_{i=1}^{d_2}\int_E \left(R_1^2\Biggl(\sum_{j=1}^{d_2}\partial_j a_{ij}\Biggr)^2+(\partial_i\Psi)^2 \right)  f^2\,\mathrm{d}(x,y) \\
	&\leq 16c_\Sigma^2M^2d_2^3\|f\|_{L^2}^2+\frac14 \int_E |\nabla\Psi|^2f^2\,\mathrm{d}(x,y).
\end{aligned}
\]

Now recall \nameref{ass:y-potential-3} and set $R_2\defeq 8c_\Sigma M_\Sigma K$; then again with Hölder and Young, but for $p=\frac2{\alpha}$, $q=\frac{2}{2-\alpha}$, we get
\[
\begin{aligned}
	\frac{R_1}2\left|\int_E \sum_{i,j=1}^{d_2} a_{ij}\partial_j\partial_i\Psi f^2\,\mathrm{d}(x,y)\right|
	&\leq \frac{R_1}2\int_E |\Sigma|_2 \cdot|\nabla^2\Psi|_2f^2\,\mathrm{d}(x,y) \\
	&\leq 2c_\Sigma M_\Sigma K \int_E (1+|\nabla\Psi|^\alpha) R_2^{-\frac{\alpha}2} f^\alpha R_2^{\frac{\alpha}2}f^{2-\alpha}\,\mathrm{d}(x,y) \\
	&\leq \frac{R_2}4\left( \|f\|_{L^2}^2+ \| |\nabla\Psi|^\alpha R_2^{-\frac{\alpha}2} f^\alpha  \|_{L^{\frac{2}{\alpha}}}  \cdot  \| R_2^{\frac{\alpha}2}f^{2-\alpha}  \|_{L^{\frac{2}{2-\alpha}}}  \right) \\
	&\leq \frac{R_2}4 \|f\|_{L^2}^2+ \frac{\alpha R_2}{8 R_2}\| |\nabla\Psi|f\|_{L^2}^2+\frac{(2-\alpha)R_2^{\frac{2}{2-\alpha}}}{8}\|f\|_{L^2}^2 \\
	&\leq \left(\frac{R_2}4+\frac{R_2^{\frac{2}{2-\alpha}}}8 \right)\|f\|_{L^2}^2 +\frac14 \int_E |\nabla\Psi|^2f^2\,\mathrm{d}(x,y).
\end{aligned}
\]
Combining these three inequalities yields \eqref{eq:Kato-bound} with
\[
	a=-8c_\Sigma \quad\text{ and }\quad
	b=2R_1^2M^2d_2^3+\frac{R_2}2+\frac{R_2^{\frac{2}{2-\alpha}}}4.
\]
\end{proof}

Since $C_c^\infty(\R^{d_1})\otimes U_\Psi C_c^\infty(\R^{d_2})$ is dense in $D_1$ wrt.~the graph norm of $L_1$,
we obtain essential m-dissipativity of $(L_1,C_c^\infty(\R^{d_1})\otimes U_\Psi C_c^\infty(\R^{d_2}))$ and therefore also of its dissipative extension $(L_1,D_2)$ with $D_2\defeq \mathcal{S}(\R^{d_1})\otimes U_\Psi C_c^\infty(\R^{d_2}))$,
where $\mathcal{S}(\R^{d_1})$ denotes the set of smooth functions of rapid decrease on $\R^{d_1}$.
Applying Fourier transform in the $x$-component leaves $D_2$ invariant and shows that
$(L_2,D_2)$ is essentially m-dissipative, where $L_2= L_0+Q\nabla\Psi\cdot\nabla_x$.
Now we add the part depending on the potential $\Phi$.

\begin{prop}
	Let $\Sigma$ satisfy \nameref{ass:coeff-growth} with $\beta=0$ and let $\Phi$ be Lipschitz-continuous.
	Then the operator $(L',D_2)$ with $L'=L_2-Q^*\nabla\Phi\nabla_y$ is essentially m-dissipative on $L^2(E,\mathrm{d}(x,y))$.
\end{prop}
\begin{proof}
	It holds due to antisymmetry of $Q\nabla\Psi\nabla_x$ that
	\[
		\|Q^*\nabla\Phi\nabla_yf\|_{L^2}^2 \leq \||Q^*\nabla\Phi|\|_\infty^2c_\Sigma\biggl((\nabla_yf,\Sigma\nabla_y f)_{L^2}+ \Bigl(\frac{\langle\nabla\Psi,\Sigma \nabla\Psi\rangle}4 f -Q\nabla\Psi\nabla_x f,f\Bigr)_{L^2} \biggr),
	\]
	which analogously to the proof of \Cref{prop:no-potential} again implies
	that the antisymmetric, hence dissipative operator $(\nabla\Phi\nabla_y,D_2)$ is $L_2$-bounded with bound zero.
	This shows the claim.
\end{proof}

Denote by $H_c^{1,\infty}(\R^{d_1})$ the space of functions in $H^{1,\infty}(\R^{d_1},\mathrm{d}x)$ with compact support
and set $D'\defeq H_c^{1,\infty}(\R^{d_1})\otimes U_\Psi C_c^\infty(\R^{d_2})$. As $(L',D')$ is dissipative and its closure extends $(L',D_2)$, it is itself essentially m-dissipative. The unitary transformation $U^{-1}$ from the beginning of this section transforms $D'$ into $D_3\defeq H_c^{1,\infty}(\R^{d_1})\otimes C_c^\infty(\R^{d_2})$, and it holds that $U^{-1}L'U = L$ on $D_3$.
Since $|\nabla\Psi|\in L_{\mathrm{loc}}^2(\R^{d_2},\mu_2)$, by the same argument as in \cite[Theorem 4.1]{BG21} we get the following result:

\begin{thm}\label{thm:generator-nice-coeff}
	Let $\Sigma$ satisfy \nameref{ass:coeff-growth} with $\beta=0$ and $\Phi$ be Lipschitz-continuous.
	Then $(L,\mathcal{C})$ is essentially m-dissipative on $H$.
\end{thm}

\subsection{Essential m-dissipativity for locally Lipschitz coefficients and potential}

Now we extend the above result to the general statement given in \Cref{thm:ess-m-diss}.
The proof is an adaptation of the procedure employed for \cite[Theorem 3.4]{BG21}.
Without loss of generality assume $\Phi\geq 0$.
For $n\in\N$ we define $\Sigma_n$ via
\[
	\Sigma_n=(a_{ij,n})_{1\leq i,j\leq d}, \quad	a_{ij,n}(y)\defeq a_{ij}\left(\left(\frac{n}{|y|}\wedge 1\right)y\right).
\]
Then each $\Sigma_n$ coincides with $\Sigma$ on $B_n(0)$ and satisfies \nameref{ass:ellipticity} with $c_{\Sigma_n}=c_\Sigma$.
For $y\in B_n(0)$, \nameref{ass:coeff-growth} implies $|\partial_k a_{ij,n}|\leq 2Mn^\beta$ for all $1\leq k\leq d_2$.
For $y\in \R^{d_2}\setminus\overline{B_n(0)}$, the chain rule suggests
\[
	|\partial_ka_{ij,n}|=\left| \partial_ka_{ij}\left(\frac{yn}{|y|}\right)-\sum_{\ell=1}^{d_2}\partial_\ell a_{ij}\left(\frac{yn}{|y|}\right)\frac{ny_k y_\ell}{|y|^3} \right|
	\leq (\sqrt{d_2}+1)Mn^\beta.
\]
Hence, $\Sigma_n$ satisfies \nameref{ass:coeff-growth} with $\beta_n= 0$ and $M_n\defeq (\sqrt{d_2}+1)Mn^\beta$.

Let further $\eta_m\in C_c^\infty(\R^{d_1})$ for each $m\in\N$ with $\eta=1$ on $B_m(0)$ and set $\Phi_m=\eta_m \Phi$, which is Lipschitz-continuous. Then set $\mu_{1,m}\defeq \mathrm{e}^{-\Phi_m}\mathrm{d}x$, $X_m \defeq L^2(E,\mu_{1,m}\otimes\mu_2)$ and define $(L_{n,m},\mathcal{C})$ via
\[
	L_{n,m}f = \sum_{i,j=1}^d a_{ij,n}\partial_{y_j}\partial_{y_i}f + \sum_{i=1}^d \sum_{j=1}^d(\partial_j a_{ij,n}-a_{ij,n}\partial_j\Psi)\partial_{y_i}f + Q\nabla\Psi\cdot\nabla_x f - Q^*\nabla\Phi_m\cdot\nabla_y f. 
\]
Then by \Cref{thm:generator-nice-coeff}, for each $n,m\in\N$, $(L_{n,m},\mathcal{C})$ is essentially m-dissipative on $X_m$,
and it holds that $L_{n,m}f = Lf$ for all $f\in \mathcal{C}$ on $B_m(0)\times B_n(0)$.
Note further that $\|\cdot\|_X\leq \|\cdot\|_{X_m}$.

Now we prove the analogue to \cite[Lemma 4.11]{BG21}:
\begin{lem}\label{lem:technical-estimate}
Let $n,m\in\N$ and $\Sigma_n$, $\Phi_m$ as defined above. Then there is a constant $D_1<\infty$
independent of $n,m$ such that for each $1\leq j\leq d_2$, the following hold for all $f\in \mathcal{C}$:
\begin{align*}
	\|\partial_j\Psi f\|_{X_m} &\leq D_1 n^{\frac{\beta}2} \|(I-L_{n,m})f  \|_{X_m},\\
	\|\partial_{y_j}f\|_{X_m} &\leq D_1 n^{\frac{\beta}2}  \|(I-L_{n,m})f  \|_{X_m}.
\end{align*}
\end{lem}
\begin{proof}
Define the unitary transformations $U_m:X_m\to L^2(E,\mathrm{d}(x,v))$ analogously to \eqref{eq:unitary-transforms},
as well as the operator $L_{n,m}'=U_mL_{n,m}U_m^{-1}$,
and let $f\in U_m \mathcal{C}=C_c^\infty(\R^{d_1})\otimes U_\Psi C_c^\infty(\R^{d_2})$. Then
\[
\begin{aligned}
	L_{n,m}'f= \sum_{i,j=1}^d a_{ij,n}\partial_{y_j}\partial_{y_i}f
	&- \frac14(\nabla\Psi,\Sigma_n \nabla\Psi)f
	+ \frac12 \sum_{i,j=1}^{d_2} a_{ij,n}\partial_j\partial_i\Psi f\\
	&+ \sum_{i,j=1}^d \partial_j a_{ij,n} (\tfrac12\partial_i\Psi f+\partial_{y_i}f)
	-Q\nabla\Psi\cdot\nabla_xf+Q^*\nabla\Phi_m\cdot\nabla_yf.
\end{aligned}
\]
Analogously to the proof of \Cref{prop:no-potential}
and due to antisymmetry of $Q\nabla\Psi\nabla_x$ and $Q^*\nabla\Phi_m\nabla_v$ on $L^2(\mathrm{d}(x,y))$, it holds that
\begin{equation}\label{eq:approx-kato-bound}
\begin{aligned}
	\| \partial_j\Psi U_m^{-1}f\|_{X_m}^2 &= \|\partial_j\Psi f\|_{L^2(\mathrm{d}(x,v))}^2 \leq 4 c_{\Sigma}\int_E\frac14 \langle y,\Sigma_n y\rangle f^2 + \langle\nabla_y f,\Sigma_n\nabla_y f\rangle\,\mathrm{d}(x,v)\\
	&\leq a(-L_{n,m}'f,f)_{L^2}+b_n\|f\|_{L^2}^2,
\end{aligned}
\end{equation}
where $a=-8c_\Sigma$ and
\[
	b_n= 2R_1^2(M_n)^2d_2^3+\frac{R_2}2+\frac{R_2^{\frac{2}{2-\alpha}}}4
	\leq 2R_1^24M^2n^{2\beta}d_2^4+\frac{R_2}2+\frac{R_2^{\frac{2}{2-\alpha}}}4.
\]

Since by the Hölder and Young inequalities, combined with dissipativity of $(L_{n,m}',U_m\mathcal{C})$, it holds that
\[
\begin{aligned}
	(-L_{n,m}'f,f)_{L^2}+\|f\|_{L^2}^2
	&= ((I-L_{n,m}')f,f)_{L^2} \leq \frac14 \left(\|(I-L_{n,m}')f \|_{L^2}+\|f\|_{L^2}\right)^2 \\
	&\leq \frac14 \left( 2\|(I-L_{n,m}')f \|_{L^2}\right)^2 = \|(I-L_{n,m}')f \|_{L^2}^2,
\end{aligned}
\]
the estimate \eqref{eq:approx-kato-bound} implies the existence of some $D_1<\infty$ such that
\[
	\| \partial_j\Psi U_m^{-1}f\|_{X_m}\leq D_1n^{\beta}\|(I-L_{n,m}')f \|_{L^2}=D_1n^{\beta}\|(I-L_{n,m})U_m^{-1}f\|_{X_m}.
\]

For the second part, note that $\partial_{y_j}U_m^{-1}f= U_m^{-1}\partial_{y_j}f+\frac12\partial_j\Psi U_m^{-1}f$ and that
\[
\begin{aligned}
	\| U_m^{-1}\partial_{y_j}f\|_{X_m}^2 &= (\partial_{y_j}f,\partial_{y_j}f)_{L^2}^2
	\leq c_\Sigma\int_{E} \langle \nabla_y f,\Sigma_n\nabla_y f\rangle + \frac14 \langle y,\Sigma_n y\rangle f^2\,\mathrm{d}(x,y)\\
	&\leq \frac14\left( a(-L_{n,m}'f,f)_{L^2}+b_n\|f\|_{L^2}^2 \right) \leq \frac14 D_1^2n^{2\beta}\|(I-L_{n,m})U_m^{-1}f\|^2_{X_m}.
\end{aligned}
\]
\end{proof}

The remainder of the proof of \Cref{thm:ess-m-diss} regarding essential m-dissipativity now follows analogously (with $\beta<\alpha<\frac1\gamma$) to the original proof of \cite[Theorem 4.3]{BG21}. By \eqref{eq:gradient-form}, $(Lf,1)_X=0$ for all $f\in\mathcal{C}$, so $\mu$ is invariant for $(L,\mathcal{C})$, which is an abstract diffusion operator. Then by \cite[Lemma 1.9]{Eberle1999}, this implies that the generated semigroup is sub-Markovian. Invariance of $\mu$ for $(L,\mathcal{C})$ also immediately implies invariance of $\mu$ for $(T_t)_{t\geq 0}$.

By considering $\widetilde{Q}\defeq -Q$ instead of $Q$ we obtain the same result for $\widetilde{L}\defeq S+A$,
which coincides with the adjoint $(L^*,D(L^*))$ of $L$ on $\core$.
Hence $(L^*,D(L^*))$ is the closure of $(S+A,\core)$ and generates the adjoint semigroup $(T_t^*)_{t\geq 0}$,
for which $\mu$ is also invariant. Now conservativity of $(T_t)_{t\geq 0}$ is immediate, which concludes the proof of \Cref{thm:ess-m-diss}.

\section{The weak hypocoercivity framework}\label{sec:weak-hypoc}

We briefly summarize the weak hypocoercivity framework as well as the main result from \cite{GW19}, which along with \cite{GS14} is recommended for more details.
Let $(H,(\cdot,\cdot),\|\cdot\|)$ be a separable Hilbert space with orthogonal decomposition $H=H_1\oplus H_2$ given by the orthogonal projection $P:H\to H_1$. Let $(L,D(L))$ be a densely defined linear operator generating a $C_0$-semigroup $(T_t)_{t\geq 0}$ on $H$ and let $D$ be a core of $(L,D(L))$. Further assume that $L$ decomposes into $L=S-A$ on $D$, where $S$ is symmetric and $A$ is antisymmetric. This implies that $(S,D)$ and $(A,D)$ are closable; we denote their closures by $(S,D(S))$ and $(A,D(A))$, respectively. We assume the following:

\begin{cond}{WH1}\label{ass:hypoc-1}
	$H_1\subset \mathcal{N}(S)\defeq \{f\in D(S) : Sf=0\}$, that is $H_1\in D(S)$ and $SP=0$, which implies $(I-P)D\subset D(S)$.
\end{cond}
\begin{cond}{WH2}\label{ass:hypoc-2}
	$PD\subset D(A)$ and $PAP|_{D}=0$.
\end{cond}
Under these assumptions, $(PA,D(A))$ is closable with closure $(PA,D(PA))$ and $AP$ with domain $D(AP)\defeq \{f\in H:Pf\in D(A)\}$ is closed and densely defined with $(AP)^*=-PA$ on $D(A)$. By von Neumann's theorem, the operators $G\defeq -(AP)^*AP$ and $I-G$ with domain $D(G)\defeq D((AP)^*AP)=\{f\in D(AP):APf\in D((AP)^*)\}$ are self-adjoint and the latter admits a bounded linear inverse. We define the operator $B$ with domain $D(B)=D((AP)^*)$ via
\begin{equation}
	B\defeq (I+(AP)^*AP)^{-1}(AP)^*.
\end{equation}
Then $B$ is bounded and extends to a bounded operator $(B,H)$ with $\|B\|\leq 1$ and $PB=B$. This allows us to formulate:
\begin{cond}{WH3}\label{ass:hypoc-3}
	We assume $D\subset D(G)$ and that there is a constant $N<\infty$ such that
	\[
	\begin{aligned}
		(BS(I-P)f,Pf)_H&\leq \frac{N}2\|(I-P)f\|_H\|Pf\|_h,\\
		-(BA(I-P)f,Pf)_H&\leq \frac{N}2\|(I-P)f\|_H\|Pf\|_h, &&f\in D.
	\end{aligned}
	\]
\end{cond}
Note that if $(AP)D\subset D(A)$, then $D\subset D(G)$ with $G=PAAP$ on $D$.
In order to ascertain the inequalities in \nameref{ass:hypoc-3} later for our application,
we use the following Lemma, see \cite[Lemma 2.3]{BG21}:
\begin{lem}\label{lem:auxiliary-bound}
	Assume that $D$ is a core of $(G,D(G))$.
	Let $(T,D(T))$ be a linear operator with $D\subset D(T)$ and assume $AP(D)\subset D(T^*)$.
	Then
	\[
		(I-G)(D)\subset D((BT)^*) \quad\text{ with }\quad (BT)^*(I-G)f = T^*APf,\quad f\in D.
	\]
	If there exists some $C<\infty$ such that
	\begin{equation}\label{eq:adjoint_bound}
		\|(BT)^* g\| \leq C\|g\|\qquad\text{ for all }g\in (I-G)D
	\end{equation}
	then $(BT,D(T))$ is bounded and its closure $(\overline{BT})$ is a bounded operator on $H$ with
	$\| \overline{BT}\|=\| (BT)^*\|$.
	
	In particular, if $(S,D(S))$ and $(A,D(A))$ satisfy these assumptions,
	the corresponding inequalities in \nameref{ass:hypoc-3}
	are satisfied with $N=2\max\{\| (BS)^*\|, \| (BA)^*\|\}$.
\end{lem}

Now let $\Theta:H\to[0,\infty]$ be a functional such that $\{f\in H:\Theta f<\infty\}$ is dense in $H$.
Denote by $(\mathrm{e}^{tG})_{t\geq 0}$ the $C_0$-semigroup generated by $G$. We assume the following:
\begin{cond}{WH4}\label{ass:hypoc-4}
	For each $f\in H$ and $t\geq 0$, the functional $\Theta$ satisfies
	\[
		\Theta(T_tf)\leq\Theta(f), \quad
		\Theta(\mathrm{e}^{tG}f)\leq\Theta(f), \quad
		\Theta(Pf)\leq\Theta(f).
	\]
\end{cond}
\begin{cond}{WH5}\label{ass:hypoc-5}
	For any $f\in D(L)$ there is some sequence $(f_n)_{n\in\N}$ in $D$ such that $f_n\to f$ in $H$ and
	\[
		\limsup_{n\to\infty}(-Lf_n,f_n)_H\leq (-Lf,f)_H,\quad
		\limsup_{n\to\infty}(\Theta f_n)\leq \Theta(f).
	\]
\end{cond}
Finally, we assume the following weak Poincaré inequalities to hold for $S$ and $A$:
\begin{cond}{WH6}\label{ass:weak-poincare}
	There exist decreasing functions $\alpha_i:(0,\infty)\to[1,\infty)$, $i=1,2$, such that
	\begin{equation}
		\|Pf\|^2\leq \alpha_1(r)\|APf\|^2+r\Theta(Pf),\quad r>0,\ f\in D(AP)
	\end{equation}
	and
	\begin{equation}
		\|(I-P)f\|^2\leq \alpha_2(r)(-Sf,f)_H+r\Theta(f),\quad r>0,\ f\in D.
	\end{equation}
\end{cond}
This allows us to state the main weak hypocoercivity result:
\begin{thm}\label{thm:abstract-hypoc}
	Let \nameref{ass:hypoc-1}-\nameref{ass:weak-poincare} be satisfied. Then there exist constants $c_1,c_2>0$ such that
	\begin{equation}
		\|T_tf\|^2\leq \xi(t)(\|f\|^2+\Theta(f)),\quad t\geq0,\ f\in D(L)
	\end{equation}
	holds for
	\begin{equation}
		\xi(t)\defeq c_1\inf\left\{ r>0:c_2 t\geq \alpha_1(r)^2\alpha_2\left( \frac{r}{\alpha_1(r)^2}\log\Bigl(\tfrac1{r}\Bigr) \right) \right\},
	\end{equation}
	which goes to $0$ as $t\to\infty$.
\end{thm}

\section{Weak hypocoercivity for generalized stochastic Hamiltonian systems with multiplicative noise}
We now embed the differential operator $L$ as defined in \Cref{def:operators} and its associated semigroup into the weak hypocoercivity framework.
We define
\[
	H\defeq \{ f\in X: \mu(f)=0\}\subset X,\quad H_1\defeq \{ f\in H: f(x,y)\text{ does not depend on }y\},
\]
as well as
\[
	P:H\to H_1,\ Pf\defeq \int_{\R^{d_2}}f(x,y)\,\mu_2(\mathrm{d}y)\quad\text{ for }f\in H
\]
and
\[
	D\defeq \{ f\in C^\infty(\R^{d_1+d_2}):\nabla f\text{ has compact support},\ \mu(f)=0  \}\subset H.
\]
Note that each $f\in D$ admits the representation $f=g-\mu(g)$ for some $g\in C_c^\infty(\R^{d_1+d_2})$.
Simple approximation shows that the closures of $(S,\mathcal{C})$, $(A,\mathcal{C})$ and $(L,\mathcal{C})$ in $X$ act trivially on constants,
which suggests the natural definition $Sf\defeq Sg$ for $f\in D$, similarly for the other operators.

Let $(L,D(L))$, $(S,D(S))$ and $(A,D(A))$ be the closures of the dissipative operators $(L,D)$, $(S,D)$ and $(A,D)$ in $H$ respectively.
Under the assumptions of \Cref{thm:ess-m-diss}, $(L,C_c^\infty(\R^{d_1+d_2}))$ is essentially m-dissipative in $X$, which implies essential m-dissipativity of $(L,D)$ in $H$.
Indeed, let $f\in H$. Then there exists a sequence $(g_n)_{n\in\N}$ in $\mathcal{C}$ such that $(I-L)g_n\to f$ in $X$ as $n\to\infty$.
Define $f_n\defeq g_n-\mu(g_n)$ for $n\in\N$, then $f_n\in D$ for all $n\in\N$ and $Lf_n=Lg_n$.
Hence $(I-L)f_n=(I-L)g_n-\mu(g_n)\to f$ in $H$ as $n\to\infty$, since $\mu(f)=0$.

We can therefore start verifying the conditions given in \Cref{sec:weak-hypoc}.
Note that Condition~\nameref{ass:hypoc-1} is fulfilled by definition, see \cite{GW19}.
Since the operator $(A,D)$ coincides with the corresponding operator in \cite{GW19},
the proofs which only concern $A$ can be repeated verbatim.
We summarize those results after assuming the following regularity for the potentials:
\begin{cond}{$\Phi$3}\label{ass:x-potential-3}
	We assume that $\Phi\in C^2(\R^{d_1})$ and that there is a constant $C<\infty$ such that
	\[
		|\nabla^2\Phi(x)|\leq C(1+|\nabla\Phi(x)|) \quad\text{ for all }x\in\R^{d_1}.
	\]
\end{cond}
\begin{cond}{$\Psi$5}\label{ass:y-potential-5}
	There is some $\psi\in C^2(\R)$ such that $\Psi(y)=\psi(|y|^2)$.
\end{cond}

Note that due to \nameref{ass:x-potential-3}, we can apply \Cref{thm:arbitrary-integrability}
to $V=\Phi$ with $g\equiv 1$ to obtain an analogue to \Cref{rem:psi-lp}.
We collate this in the following:
\begin{corollary}\label{cor:arbitrary-integrability}
	Let $\Phi$ satisfy \nameref{ass:x-potential-3}.
	Then $|\nabla\Phi|\in L^p(\mu_1)$ and $|\nabla\Psi|\in L^p(\mu_2)$ for any $1\leq p<\infty$.
	The same integrability also extends to the Hessian matrices $|\nabla^2\Phi|$ and $|\nabla^2\Psi|$ as well.
\end{corollary}

We can now summarize the results for the antisymmetric part $A$ proved in \cite{GW19}:
\begin{lem}
	If \nameref{ass:x-potential-3} and \nameref{ass:y-potential-5} are satisfied, then
	Condition \nameref{ass:hypoc-2} is fulfilled and  $(AP)D\subset D(A)$.
	Furthermore, the operator $(G,D)\defeq (PAAP,D)$, which is given for all $f\in D$ by
	\[
		Gf= \frac{\mu_2(|\nabla\Psi|^2)}{d_2}\sum_{i,j=1}^{d_1} (QQ^*)_{ij} (\partial_{x_j}\partial_{x_i}-\partial_{x_j}\Phi\,\partial_{x_i})Pf,
	\]
	is essentially self-adjoint and its closure $(G,D(G))$ generates a sub-Markovian strongly continuous semigroup on $H$.
	Moreover, there is a constant $c_A$ only depending on the choice of $\Phi$ and $\Psi$ such that
	\[
		\|(BA)^*g\|_H\leq c_A\|g\|_H\qquad\text{ for all }g\in (I-G)D.
	\]
\end{lem}
This allows us to use \Cref{lem:auxiliary-bound} to prove \nameref{ass:hypoc-3}.
Indeed, the inequality for the antisymmetric part follows immediately from the previous Lemma, so it remains to show the first inequality. For this, we introduce the final assumption on $\Psi$:
\begin{cond}{$\Psi$6}\label{ass:y-potential-6}
	$\Psi$ is three times weakly differentiable and for any $1\leq i,j,k\leq d_2$, it holds that
	$\partial_i\partial_j\partial_k\Psi\in L^2(\mu_2)$.
\end{cond}
Note that in light of \Cref{cor:arbitrary-integrability}, this holds in particular if the third order partial derivatives are dominated by a multiple of $(1+|\nabla^2\Psi|^\beta)$ for any $0<\beta<\infty$.
We also need to assume some additional integrability on $\nabla\Sigma$:
\begin{cond}{$\Sigma$4}\label{ass:mat-cond-4}
	There is some $1<p_\Sigma\leq \infty$ such that $|\nabla\Sigma|\in L^{2p_\Sigma}(\mu_2)$.
\end{cond}
\begin{lem}\label{lem:mult-op}
	Let \nameref{ass:y-potential-5}, \nameref{ass:y-potential-6} and \nameref{ass:mat-cond-4} be satisfied.
	Then the operator $T:D\to L^2(\mu)$ with
	\begin{align*}
		Tf&\defeq \sum_{i=1}^{d_1} w_i\cdot(\partial_{x_i} Pf)
		\qquad\text{ for }w_i:\R^{d_2}\to\R \quad\text{ defined by }\\
		w_i&\defeq \sum_{j,k,\ell=1}^{d_2}q_{ij} \left( 
			\partial_k a_{k\ell}(\partial_{\ell}\partial_j\Psi)
			+ a_{k\ell}(\partial_k\partial_{\ell}\partial_j\Psi)
			- a_{k\ell}(\partial_{\ell}\partial_j\Psi)(\partial_k\Psi)
		\right)
	\end{align*}
	is well-defined and there is some $R<\infty$ such that $\|Tf\|_{L^2(\mu)}^2\leq R\| \nabla f_P\|_{L^2(\mu_1)}^2$
	for all $f\in D$.
\end{lem}
\begin{proof}
	It holds that $\|Tf\|_{L^2(\mu)}^2\leq \sum_{i=1}^{d_1}\|w_i\|_{L^2(\mu_2)}^2 \|\nabla f_P\|_{L^2(\mu_1)}^2$
	and for each $1\leq i\leq d_1$:
	\[
		\|w_i\|_{L^2(\mu_2)}
		\leq |Q| \sum_{j,k,\ell=1}^{d_2}
			\|\partial_k a_{k\ell}(\partial_{\ell}\partial_j\Psi)\|_{L^2}
			+ \|a_{k\ell}(\partial_k\partial_{\ell}\partial_j\Psi)\|_{L^2}
			+ \|a_{k\ell}(\partial_{\ell}\partial_j\Psi)(\partial_k\Psi)\|_{L^2}
	\]
	Let $1\leq q_\Sigma<\infty$ be such that $\frac1{p_\Sigma}+\frac1{q_\Sigma}=1$.
	Then by \nameref{ass:mat-cond-4} and \Cref{cor:arbitrary-integrability}, it holds that
	\[
		\|\partial_k a_{k\ell}(\partial_{\ell}\partial_j\Psi)\|_{L^2}
		\leq \|(\partial_k a_{k\ell})^2\|_{L^{p_\Sigma}} \|(\partial_\ell\partial_j\Psi)^2\|_{L^{q_\Sigma}}
		< \infty.
	\]
	Boundedness of $\Sigma$ together with \nameref{ass:y-potential-6} yields
	\[
		\|a_{k\ell}(\partial_k\partial_{\ell}\partial_j\Psi)\|_{L^2}
		\leq M_\Sigma \|\partial_k\partial_{\ell}\partial_j\Psi\|_{L^2}
		<\infty,
	\]
	and by  \nameref{ass:y-potential-3} and again \Cref{cor:arbitrary-integrability}, we also get
	\[
		\|a_{k\ell}(\partial_{\ell}\partial_j\Psi)(\partial_k\Psi)\|_{L^2}
		\leq KM_\Sigma\left( \|\partial_k\Psi\|_{L^2} + \||\nabla\Psi|^{1+\alpha}\|_{L^2} \right)
		<\infty.
	\]
	This shows that each $w_i$ is in $L^2(\mu_2)$ and therefore $R\defeq \sum_{i=1}^{d_1}\|w_i\|_{L^2(\mu_2)}^2<\infty$.
\end{proof}
Now we show that $(AP)D\subset D(S^*)$.
To that end, let $f\in D$, $h\in D(S)$ and $(h_n)_{n\in\N}$ be a sequence in $D$ such that $h_n\to h$ and $Sh_n\to Sh$ in $H$ as $n\to\infty$.
Then integration by parts applied twice yields
\[
\begin{aligned}
	(Sh,APf)_H
	&=\lim_{n\to\infty}(Sh_n,-Q\nabla\Psi\cdot\nabla(Pf))_H
	=\lim_{n\to\infty}\mu(\langle\nabla_y h_n,\Sigma \nabla_y(Q\nabla\Psi\cdot\nabla_x(Pf))\rangle)\\
	&= \lim_{n\to\infty} (h_n, Tf)_{L^2(\mu)}
	= (h,Tf)_{L^2(\mu)}.
\end{aligned}
\]
Setting $h\equiv 1$ shows that $Tf\in H$ since $Sh=0$, which implies $APD\subseteq D(S^*)$.
Now by \Cref{lem:auxiliary-bound}, it follows that $(BS)^*(I-G)f=S^*APf=Tf$ and
\[
	\|Tf\|_H^2
	\leq R \|\nabla f_P\|_{L^2(\mu_1)}^2
	\leq \frac{d_2R|Q^{-1}|^2}{\mu_2(|\nabla\Psi|^2)} \frac{\mu_2(|\nabla\Psi|^2)}{d_2} \|Q^*\nabla f_P\|_{L^2(\mu_1)}^2.
\]
Further, since via integration by parts
\[
\begin{aligned}
	\|Q^*\nabla f_P\|_{L^2(\mu_1)}^2
	&= -\int_{\R^{d_1}} \sum_{i,j=1}^{d_1} (QQ^*)_{ij}\, f_P\,(\partial_j\partial_if_P-\partial_j\Phi\partial_i f_P)\ \mathrm{d}\mu_1\\
	&= -\int_{\R^{d_1+d_2}} \sum_{i,j=1}^{d_1} (QQ^*)_{ij}\, Pf\,(\partial_{x_j}\partial_{x_i}Pf-\partial_j\Phi\partial_{x_i}Pf)\ \mathrm{d}\mu \\
	&= -\frac{d_2}{\mu_2(|\nabla\Psi|^2)}\int_{\R^{d_1+d_2}} Pf Gf\,\mathrm{d}\mu
\end{aligned}
\]
we obtain for $g\defeq (I-G)f$ due to dissipativity of $(G,D)$ that
\[
\begin{aligned}
	\|(BS)^*g\|_H^2
	&=\|Tf\|_H^2
	\leq \frac{d_2R|Q^{-1}|^2}{\mu_2(|\nabla\Psi|^2)} \int_E Pf\, (-G)f\ \mathrm{d}\mu\\
	&\leq \frac{d_2R|Q^{-1}|^2}{\mu_2(|\nabla\Psi|^2)} \|Pf\|_H (\|(I-G)f\|_H + \|f\|_H) \\
	&\leq \frac{2d_2R|Q^{-1}|^2}{\mu_2(|\nabla\Psi|^2)} \|(I-G)f\|_H^2 
	= \frac{2d_2R|Q^{-1}|^2}{\mu_2(|\nabla\Psi|^2)} \|g\|_H^2.
\end{aligned}
\]
Using the second part of \Cref{lem:auxiliary-bound}, this shows that $(BS,D(S))$ is bounded and hence \nameref{ass:hypoc-3} is fulfilled.

\begin{defn}
Let the functional $\Theta:H\to[0,\infty]$ be defined by \[\Theta f\defeq \|f\|_\mathrm{osc}^2=(\esssup f-\essinf f)^2\]
for all $f\in H$.
\end{defn}
Then the set $\{f\in H:\Theta f<\infty\}\supset D$ is dense in $H$ and \nameref{ass:hypoc-4} is satisfied due to $L^\infty$-contractivity of the semigroups.

We show \nameref{ass:hypoc-5} using the same construction as in \cite{GW19}: Fix some $f\in D(L)$ and set $\gamma_1\defeq\essinf f$, $\gamma_2\defeq \esssup f$.
Let $(g_n)_{n\in\N}$ be a sequence in $D$ such that $g_n\to f$ and $Lg_n\to Lf$ in $H$ as $n\to \infty$. Then set $f_n\defeq h_n(g_n)$, where $h_n\in C^\infty(\R)$ satisfies
$0\leq h_n'\leq 1$ and 
\[
	h_n(r)=\begin{cases}
		r & \text{ for }r\in[\gamma_1,\gamma_2]\\
		\gamma_1-\frac1{2n} & \text{ for }r\leq\gamma_1-\frac1{n}\\
		\gamma_2+\frac1{2n} & \text{ for }r\geq\gamma_2+\frac1{n}
		\end{cases}.
\]
Then $f_n\to f$ in $H$ as $n\to\infty$, $\limsup_{n\to\infty}\|f_n\|_\mathrm{osc}\leq \|f\|_\mathrm{osc}$, and
\[
\begin{aligned}
	\limsup_{n\to\infty} (-Lf_n,f_n)
	&=\limsup_{n\to\infty} \mu(\langle\nabla_y f_n,\Sigma \nabla_y f_n\rangle)
	=\limsup_{n\to\infty} \mu((h_n'(g_n))^2 \langle\nabla_y g_n,\Sigma \nabla_y g_n\rangle)\\
	&\leq \limsup_{n\to\infty} \mu(\langle\nabla_y g_n,\Sigma \nabla_y g_n\rangle)
	=\limsup_{n\to\infty} (-Lg_n,g_n)
	= (-Lf,f).
\end{aligned}
\]
Finally, it remains to show the weak Poincaré inequalities are satisfied. For this, we use \cite[Theorem 3.1]{RW01}, which for our purposes states that, given
a probability measure $\mu_V=\mathrm{e}^V\mathrm{d}x$ on $\R^d$ with locally bounded $V$, there exists a decreasing function $\alpha:(0,\infty)\to(0,\infty)$ such that
\begin{equation}\label{eq:wp-base}
	\mu_V(f^2)\leq \alpha(r)\mu_V(|\nabla f|^2)+r\|f\|_\mathrm{osc}^2,\quad r>0,f\in C_b^1(\R^d).
\end{equation}
The first inequality in \nameref{ass:weak-poincare} is already proved in \cite{GW19}, which also gives a procedure to show the second one:
Let $f\in D$ and $x\in\R^{d_1}$, set $f_x\defeq f(x,\cdot)-Pf(x)\in C^\infty(\R^{d_2})$.
Then $\nabla f_x$ has compact support, so that $f_x$ is bounded; $\mu_2(f_x)=0$ and $\|f_x\|_\mathrm{osc}\leq \|f\|_\mathrm{osc}$.
Therefore, \eqref{eq:wp-base} is applicable and yields the existence of a decreasing function $\alpha_\Psi$ such that
\[
\begin{aligned}
	\mu_2(f_x^2)
	&\leq \alpha_\Psi(r)\mu_2(|\nabla_y f(x,\cdot)|^2)+r\|f_x\|_\mathrm{osc}^2\\
	&\leq c_\Sigma\alpha_\Psi(r)\mu_2(\langle\nabla_y f(x,\cdot),\Sigma\nabla_y f(x,\cdot)\rangle)+r\|f\|_\mathrm{osc}^2
\end{aligned}
\]
for all $r>0$, $f\in D$ and $x\in\R^{d_1}$.
Integrating that expression wrt.~$\mu_1$ gives
\[
\begin{aligned}
	\|(I-P)f\|_H^2
	&=\int_{\R^{d_1}}\mu_2(f_x^2)\,\mu_1(\mathrm{d}x)
	\leq  c_\Sigma\alpha_\Psi(r) \int_E \langle\nabla_y f,\Sigma\nabla_y f\rangle\,\mu(\mathrm{d}(x,y))+r\|f\|_\mathrm{osc}^2\\
	&= c_\Sigma\alpha_\Psi(r) (-Sf,f)_H +r\Theta f
\end{aligned}
\]
for all $f\in D$, $r>0$, so the second inequality in \nameref{ass:weak-poincare} is satisfied with $\alpha_2\defeq c_\Sigma\alpha_\Psi$.
\begin{proof}[of \texorpdfstring{\Cref{thm:main-result}}{Theorem 1.3}]
	Let the assumptions of \Cref{thm:main-result} hold.
	Then clearly all assumptions of \Cref{thm:ess-m-diss} are satisfied (see \Cref{cor:arbitrary-integrability}), which yields the existence of the $C_0$-semigroup $(T_t)_{t\geq 0}$ on $L^2(\mu)$ with the stated properties. The existence of the decreasing functions $\alpha_\Phi$ and $\alpha_\Psi$ is again provided by \cite[Theorem 3.1]{RW01}.
	
	In order to satisfy \nameref{ass:y-potential-5}, we first shift the $y$-variable by $a$ and then consider $\overline{\Psi}(y)\defeq \Psi(\Lambda^{-1}y)=\psi(|y|^2)$, with $\overline{\mu_2}$ and $\overline{\mu}$ defined accordingly. Let $J:L^2(\mu)\to L^2(\overline{\mu})$ be given by $Jf(x,y)\defeq \sqrt{|\det(\Lambda)|}^{-1}f(x,\Lambda^{-1}y)$.
	Then $J$ is a unitary transformation which leaves $D$ invariant, and we define the operator $(\overline{L},D)$ by $\overline{L}\defeq JLJ^{-1}$.
	This operator has the same structure as $L$, except with $\overline{Q}=Q\Lambda^*$ and $\overline{\Sigma}(y)=\Lambda\Sigma(\Lambda^{-1}y)\Lambda^*$.
	So wlog we can assume that already $\Psi(y)=\psi(|y|^2)$. We then restrict the setting to $H\subseteq X$ and consider the closure $(L^H,D(L^H))$ of $(L,D)$ in $H$
	with the corresponding semigroup $(T_t^H)_{t\geq 0}$ on $H$.
	
	It can then easily be checked with the previous considerations in this section that all conditions for \Cref{thm:abstract-hypoc} are met, which gives the existence of constants $0<\tilde{c}_1,c_2<\infty$ such that
	\[
		\mu((T_t^H f)^2)\leq \tilde{c}_1\xi(t)(\mu(f^2)+\|f\|_{\mathrm{osc}}^2),
		\qquad t\geq 0,\ f\in D(L^H)
	\]
	with $\xi$ given as in the statement of \Cref{thm:main-result}.
	Since $L$ and $L^H$ coincide on $D$, the semigroup $(T_t)_{t\geq 0}$ coincides with $(T_t^H)_{t\geq 0}$ on $H$.
	Invariance and conservativity imply that
	\[
		T_t(\mu(f))=\mu(f)T_t(1)=\mu(f)=\mu(T_tf)
		\quad\text{ for all }f\in L^2(\mu).
	\]
	Since $f-\mu(f)\in D(L^H)$ for all $f\in D(L)$, this yields
	\[
		\mu((T_tf)^2)-\mu(T_tf)^2
		=\mu((T_t^H(f-\mu(f)))^2)
		\leq \tilde{c}_1\xi(t)(\mu((f-\mu(f))^2)+\|f\|_{\mathrm{osc}}^2)
	\]
	for all $f\in D(L)$ and $t\geq 0$.
	Since $|f-\mu(f)|\leq \|f\|_\mathrm{osc}$, this shows the wanted estimate with $c_1=2\tilde{c}_1$ for all $f\in D(L)$.
	This extends to $f\in L^\infty(\mu)$ by approximation as in the proof of \nameref{ass:hypoc-5}.
\end{proof}

\section{Applications}

We briefly present direct applications of the hypocoercivity result to PDEs and SDEs.
Since the principles used are the same, we refer to \cite{BG21} for details.
We assume the conditions stated in \Cref{thm:main-result}.

\subsection{The abstract Cauchy problem}\label{sec:cauchy}
We consider the abstract Cauchy problem associated with the operator $L$ on $X=L^2(\mu)$:
\begin{equation}\label{eq:cauchy-kol}
	u(0)=u_0, \qquad
	\partial_t u(t) = Lu(t), \quad \text{ for all }t\geq 0,
\end{equation}
where $(L,D(L))$ is the closure of $(L,\mathcal{C})$ on $X$, $u:[0,\infty)\to X$ and $u_0\in X$.

Let $(T_t)_{t\geq 0}$ denote the semigroup on $X$ generated by $(L,D(L))$
and set $u(t)\defeq T_tu_0$. Then $u$ is the unique mild solution to \eqref{eq:cauchy-kol} for all $u_0\in X$
and the unique classical solution for all $u_0\in D(L)$, so in particular for all $u_0\in \{ f\in C^2(\R^{d_1+ d_2}):\nabla f\text{ has compact support}\}$.

Furthermore, \Cref{thm:main-result} shows that for all bounded measurable $u_0$,
we obtain convergence of the unique solution $u(t)$ to a constant as $t\to\infty$. More precisely, there are explicit constants $c_1,c_2>0$ depending on the choice of $\Sigma$, $\Phi$ and $\Psi$ such that for all $t\geq 0$,
\[
	\left\| u(t)-\int_E u_0\,\mathrm{d}\mu \right\|_X^2 \leq c_1\xi(t) \left\|u_0 \right\|_\mathrm{osc}^2,
\]
with $\xi$ given by \eqref{eq:conv-rate}.

\subsection{The associated Fokker-Planck equation in divergence form}\label{sec:fokker-planck}

Consider the Hilbert space $\widetilde{X}\defeq L^2(E;\widetilde{\mu})$, where
\[
	\widetilde{\mu}\defeq Z(\Phi)^{-\frac12}Z(\Psi)^{-\frac12}
	\mathrm{e}^{\Phi(x)+\Psi(y)}\,\mathrm{d}x\,\mathrm{d}y,
\]
and define the differential operator $(L^\mathrm{FP},\mathcal{C})$ on $\widetilde{X}$ for all $f\in \mathcal{C}=C_c^\infty(\R^{d_1})\otimes C_c^\infty(\R^{d_2})$ as in \eqref{eq:proto-fokker-planck}.
The abstract Cauchy problem associated with the closure $(L^\mathrm{FP},D(L^\mathrm{FP}))$ is given by
\eqref{eq:cauchy-fp}.
Using the unitary transformation
\[
T:X\to \widetilde{X},\quad  Tg=\rho g \quad\text{ with }\quad \rho(x,v)\defeq \mathrm{e}^{-\Phi(x)-\Psi(y)},
\]
as well as the result on the adjoint semigroup $(T_t^*)_{t\geq 0}$ stated at the end of the proof of \Cref{thm:ess-m-diss},
we obtain essential m-dissipativity of $(L^\mathrm{FP}, \mathcal{D})$,
where $\mathcal{D}\defeq C_c^2(E)$ is a $T$-invariant core of $(L^*,D(L^*))$.
A simple approximation argument gives essential m-dissipativity of $(L^\mathrm{FP},\mathcal{C})$ on $\widetilde{X}$, and the semigroup $(\widetilde{T}_t)_{t\geq 0}=(T(T_t^*)T^{-1})_{t\geq 0}$ generated by the closure $(L^\mathrm{FP},D(L^\mathrm{FP}))$ induces solutions of the Cauchy problem as seen in \Cref{sec:cauchy}.

Let $u_0\in\widetilde{X}$ such that $\frac{u_0}{\rho}$ is bounded and set $u(t)\defeq \widetilde{T}_t u_0$. Then \Cref{thm:main-result} yields the existence of constants $c_1$, $c_2$ depending on $\Sigma$, $\Phi$ and $\Psi$ such that for all $t\geq 0$,
\[
\begin{aligned}
	\left\| u(t)-\rho\int_E u_0\rho\,\mathrm{d}\widetilde{\mu} \right\|_{\widetilde{X}}^2
	&= \left\| T(T_t^*)T^{-1}u_0-T(1)\int_E \frac{u_0}{\rho}\,\mathrm{d}\mu \right\|_{\widetilde{X}}^2\\
	&= \left\| (T_t^*)T^{-1}u_0-\int_E T^{-1}u_0\,\mathrm{d}\mu \right\|_X^2
	\leq c_1 \xi(t) \left\|\frac{u_0}{\rho} \right\|_\mathrm{osc}^2,
\end{aligned}
\]
with $\xi$ given by \eqref{eq:conv-rate}. This shows convergence of the solution $u$ to the stationary solution $\rho\int_E u_0\,\mathrm{d}(x,y)$ in $\widetilde{X}$ as $t\to\infty$ with the rate $\xi$.

\subsection{Degenerate Diffusion with multiplicative noise}

Finally, we connect the convergence result to the initially stated Itô SDE \eqref{eq:sde}.
Under the assumptions of \Cref{thm:ess-m-diss}, the theory of Generalized Dirichlet Forms developed by Stannat (see \cite{GDF}) yields the existence of a Hunt process
\[
	\mathbf{M} = (\Omega,\mathcal{F},(\mathcal{F}_t)_{t\geq 0},(X_t,Y_t), (P_{(x,y)})_{(x,y)\in\R^{d_1}\times\R^{d_2}})
\]
with state space $E=\R^{d_1}\times\R^{d_2}$ and $P_{(x,y)}$-a.s. continuous sample paths for all $(x,y)\in E$,
which is properly associated in the resolvent sense with the semigroup $(T_t)_{t\geq 0}$ generated by $(L,D(L))$.
In particular, we obtain that for each $f\in\widetilde{X}$ and all $t>0$, $T_tf$ is a $\mu$-version of $p_tf$, where $(p_t)_{t\geq 0}$ is the transition semigroup of $\mathbf{M}$ given by
\[
	p_tf:\R^{d_1}\times\R^{d_2}\to \R, \qquad
	(x,y)\mapsto \mathbb{E}_{(x,y)}[f(X_t,Y_t)].
\]
Due to \cite[Proposition~1.4]{BBR06}, for any $\mu$-probability density $h\in L^1(\mu)$ (including $h\equiv 1$),
the probability law $P_h\defeq \int_E P_{(x,y)}h(x,y)\,\mathrm{d}\mu$ on the space of continuous paths solves the martingale problem for $(L,D(L))$ on $C_c^2(E)$, i.e. for each $f\in C_c^2(E)$, the stochastic process $(M_t^{[f]})_{t\geq 0}$ given by
\[
	M_t^{[f]}\defeq f(X_t,Y_t)-f(X_0,Y_0)-\int_0^t Lf(X_s,Y_s)\,\mathrm{d}s
\]
is a martingale with respect to $P_h$. Additionally, by \cite[Lemma~2.1.8]{Conrad2011}, for $f\in D(L)$ with $f^2\in D(L)$ and $Lf\in L^4(\mu)$, the process $(N_t^{[f]})_{t\geq 0}$ defined by
\[
	N_t^{[f]}\defeq (M_t^{[f]})^2-\int_0^tL(f^2)(X_s,Y_s)-2(fLf)(X_s,Y_s)\,\mathrm{d}s,
	\quad t\geq 0,
\]
is also a martingale wrt.~$P_h$. Analysis of the coordinate processes together with Lévy's characterization
of Brownian motion shows that $(X_t,Y_t)_{t\geq 0}$ is a weak solution to the SDE \eqref{eq:sde}, since $\sigma=\sqrt{\Sigma}$ is invertible with locally bounded inverse.

Under the assumptions of \Cref{thm:main-result}, we therefore obtain the existence of constants $c_1,c_2$ depending on $\Sigma$, $\Phi$ and $\Psi$ such that
\[
	\left\| p_tf - \mu(f) \right\|_{L^2(E;\mu)}^2 \leq c_1 \xi(t) \|f\|_{\mathrm{osc}}^2
\]
for all bounded measurable $f:E\to\R$. This implies that the law $P_1$ with initial distribution $\mu$ is strongly mixing, i.e. for any two Borel sets $A_1,A_2$ on the space of continuous paths, it holds
\[
	P_1(\varphi_t A_1\cap A_2) \to P_1(A_1) P_1(A_2)
	\quad \text{ as }t\to\infty,
\]
where $\varphi_tA_1$ denotes the set of shifted paths
\[
	\varphi_t A_1 \defeq \{ (Z_s)_{s\geq 0}\in C([0,\infty),E) \mid (Z_{s+t})_{s\geq 0}\in A_1 \}.
\]

\bibliography{sources}
\bibliographystyle{numbered}

\end{document}